\documentclass[12pt]{article}
\usepackage[english]{babel}
\usepackage[all]{xy}
\usepackage{fancyhdr}
\usepackage{setspace, amsmath, amsthm, amssymb, amsfonts, amscd, epic, graphicx, ulem, dsfont}
\usepackage[bookmarks=false]{hyperref}
\newtheorem{theorem}{Theorem}

\newtheorem{proposition}[theorem]{Proposition}
\newtheorem{corollary}[theorem]{Corollary}
\newtheorem{example}[theorem]{Example}

\newtheorem{remark}[theorem]{Remark}

\makeatletter

\@addtoreset{equation}{section}
\makeatother

\title{On the generalization of biharmonic hypersurfaces and biharmonic curves}

\author{Moustafa Tadj \footnote{Laboratory of Mathematics, Statistics and Computer Science for scientific research (W1550900), University of Naama, Algeria. Email: tadj.moustafa@cuniv-naama.dz},  Ahmed Mohammed Cherif\footnote{University Mustapha Stambouli Mascara, Faculty of Exact Sciences, Mascara 29000, Algeria. Email: a.mohammedcherif@univ-mascara.dz}, and Fethi Latti\footnote{Laboratory of Mathematics, Statistics and Computer Science for scientific research (W1550900), University of Naama, Algeria. Email: etafati@hotmail.fr}}
\date{}

\begin{document}
	\maketitle
	
\begin{abstract}
In this work, we extend the concepts of $p$-biharmonic maps and $p$-biharmonic hypersurfaces to provide a broader characterization of $(p,q)$-harmonic hypersurfaces and $(p,q)$-harmonic curves in Riemannian manifolds, including Einstein spaces. Moreover, we present new explicit examples of proper $(p,q)$-harmonic hypersurfaces and $(p,q)$-harmonic curves in space forms.\\
		\begin{flushleft}
		\end{flushleft}
\end{abstract}
	
\section{Introduction}

$p$-harmonic maps $\varphi:(M,g)\to (N,h)$ between Riemannian manifolds are defined as the critical points of the $p$-energy functional
\begin{equation}\label{eq1.1A}
E_p(\varphi;D)=\frac{1}{p}\int_{D} |d\varphi|^p\, v_g ,
\end{equation}
for any compact domain $D\subset M$ and for a constant $p>1$.
The associated Euler--Lagrange equation corresponding to \eqref{eq1.1A} is given by the vanishing of the $p$-tension field
\begin{equation}\label{eq1.1B}
\tau_p(\varphi)=\operatorname{div}^M\big(|d\varphi|^{p-2} d\varphi\big).
\end{equation}
A smooth map $\varphi$ is said to be $p$-harmonic if and only if $\tau_p(\varphi)=0$ (for more details on the concept of $p$-harmonic maps see \cite{BG,BI,ali}). Thus, $p$-harmonic maps can be viewed as a natural generalization of harmonic maps.\\
The notion of $p$-biharmonic maps extends the classical theory of harmonic and biharmonic maps (see \cite{Jiang, YX}).
A map $\varphi:(M,g)\to (N,h)$ between Riemannian manifolds is called $p$-biharmonic if it is a critical point of the $p$-biharmonic energy functional defined by
\begin{equation}\label{eq1.1}
E_{2,p}(\varphi;D)=\frac{1}{p}\int_{D} |\tau(\varphi)|^p\, v_g ,
\end{equation}
where $\tau(\varphi)$ denotes the tension field of $\varphi$, given by
$\tau(\varphi)=\operatorname{trace}_g\nabla d\varphi$.\\
In \cite{cherif2}, the author introduced the concept of bi-$p$-harmonic maps as follows.
For a smooth map $\varphi:(M,g)\to (N,h)$, the $p$-bienergy functional is defined by
\begin{equation}\label{eq1.2}
E_{p,2}(\varphi;D)=\frac{1}{2}\int_{D} |\tau_p(\varphi)|^2\, v_g .
\end{equation}
The map $\varphi$ is called bi-$p$-harmonic if it is a critical point of the $p$-bienergy functional \eqref{eq1.2} for every compact domain $D\subset M$.\\	
In this paper, we further generalize the notions of $p$-biharmonic and bi-$p$-harmonic maps by introducing the concept of $(p,q)$-harmonic maps. A smooth map $\varphi:(M,g)\to (N,h)$ between Riemannian manifolds is called $(p,q)$-harmonic if it is a critical point of the $(p,q)$-energy functional
\begin{equation}\label{eq1.3}
E_{p,q}(\varphi;D)=\frac{1}{q}\int_{D} |\tau_{p}(\varphi)|^{q}\, v_{g},
\end{equation}
for any compact domain $D\subset M$, where $p,q>1$ are constants.
Observe that when $p=q=2$, the notion of $(p,q)$-harmonic maps reduces to that of biharmonic maps.\\
Let $M$ be an $m$-dimensional submanifold of a Riemannian manifold $(N,h)$, and let
$\mathbf{i}:M\hookrightarrow (N,h)$ denote the canonical inclusion.
We denote by $g$ the Riemannian metric on $M$ induced by $h$.
Let $\nabla^N$ (resp. $\nabla^M$) be the Levi-Civita connection on $(N,h)$ (resp. $(M,g)$),
and let $\nabla^\perp$ denote the normal connection of $(M,g)$ in $(N,h)$.
We denote by $H$ the mean curvature vector field and by $A$ the shape operator of $(M,g)$ in $(N,h)$ (see \cite{baird,ON}).\\
The tension field of the inclusion $\mathbf{i}$ is given by
$\tau(\mathbf{i}) = mH.$
A submanifold $(M,g)$ is said to be $p$-harmonic (resp. $(p,q)$-harmonic) in $(N,h)$ if the tension field satisfies $\tau_p(\mathbf{i})=0$ (resp. the $(p,q)$-tension field satisfies $\tau_{p,q}(\mathbf{i})=0$).
Non $p$-harmonic $(p,q)$-harmonic submanifolds are called proper $(p,q)$-harmonic.\\
In \cite{OU}, Ye-Lin Ou proved that a hypersurface $(M,g)$ in a Riemannian manifold $(N,h)$ with mean curvature vector field $H=f\eta$, where $\dim N=m+1$, is biharmonic if and only if
\begin{equation}\label{S}
\left\{
\begin{array}{lll}
-\Delta^M(f)+f|A|^2-f\,\mathrm{Ric}^N(\eta,\eta) &=& 0, \\\\
2A(\mathrm{grad}^M\,f)+mf\,\mathrm{grad}^M\,f-2f(\mathrm{Ricci}^N\,\eta)^\top &=& 0,
\end{array}
\right.
\end{equation}
where $\mathrm{Ric}^N$ (resp. $\mathrm{Ricci}^N$) denotes the Ricci curvature (resp. the Ricci tensor) of $(N,h)$.
Here, $f$ denotes the mean curvature function of $(M,g)$ and $A$ is the shape operator with respect to the unit normal vector field $\eta$.\\
In \cite{NU,OU}, the authors studied biharmonic hypersurfaces in Riemannian manifolds and investigated several properties of such hypersurfaces, particularly in the case where the ambient manifold has non-positive Ricci curvature.
 In \cite{MM2}, the authors presented a generalization of these results. \\
 In this work, we derive the first variation formula of the $(p,q)$-energy functional and construct a model for $(p,q)$-harmonic hypersurfaces. Furthermore, we present new examples of proper $(p,q)$-harmonic hypersurfaces. These notions provide a unified framework that generalizes several energy functionals arising in geometric analysis and are closely related to the study of nonlinear differential equations. Finally, we investigate the case of curves in three-dimensional space forms.	
\section{The $(p,q)$-Harmonic maps}
In \cite{LC}, the authors provide an explicit formula for the first variation of the $(p,q)$-energy functional and define the corresponding $(p,q)$-tension field $\tau_{p,q}(\varphi)$ as follows.
\begin{theorem}\label{th01}
Let $\varphi$ be a smooth map from a Riemannian manifold $(M,g)$ to a Riemannian manifold $(N,h)$, and let $\{\varphi_t\}_{t\in(-\varepsilon,\varepsilon)}$ be a smooth variation of $\varphi$ with compact support in a domain $D\subset M$. Then
\begin{equation}\label{equa01}
\frac{d}{dt}E_{p,q}(\varphi_t;D)\Big|_{t=0}
=-\int_D h\big(v,\tau_{p,q}(\varphi)\big)\,v_g,
\end{equation}
where $\tau_{p,q}(\varphi)$ is the $(p,q)$-tension field of $\varphi$ given by
\begin{eqnarray*}
\tau_{p,q}(\varphi)
&=&-|d\varphi|^{p-2}|\tau_{p}(\varphi)|^{q-2}\operatorname{trace}_g
R^{N}\big(\tau_{p}(\varphi),d\varphi\big)d\varphi\\
&&- \operatorname{trace}_g\nabla^{\varphi}|d\varphi|^{p-2}
\nabla^{\varphi}|\tau_{p}(\varphi)|^{q-2}\tau_{p}(\varphi)\\
&&-(p-2)\operatorname{trace}_g\nabla^{\varphi}|d\varphi|^{p-4}
\left\langle\nabla^{\varphi}|\tau_{p}(\varphi)|^{q-2}\tau_{p}(\varphi),d\varphi\right\rangle d\varphi,
\end{eqnarray*}
and $v=\frac{d\varphi_t}{dt}\big|_{t=0}$ denotes the variation vector field of $\{\varphi_t\}$.
The map $\varphi$ is $(p,q)$-harmonic if its associated $(p,q)$-tension field vanishes.
\end{theorem}
\begin{proof}
Let $\phi$ be the smooth map defined by
\begin{eqnarray*}
\phi:(-\epsilon,\epsilon)\times M &\longrightarrow& N, \\
(t,x) &\longmapsto& \phi(t,x)=\varphi_{t}(x).
\end{eqnarray*}
Let $\{e_i\}_{i=1}^m$ be a local orthonormal frame on $(M,g)$ such that
$\nabla^{M}_{e_i}e_i=0$ at a point $x\in M$, for $i=1,\ldots,m$.
Note that $\phi(0,x)=\varphi(x)$. The variation vector field $v$ associated with the variation $\{\varphi_t\}_{t\in(-\epsilon,\epsilon)}$ is given by
\[
v=d\phi\!\left(\frac{\partial}{\partial t}\right)\Big|_{t=0}.
\]
Therefore, we have
\begin{equation}\label{equa005}
\frac{d}{dt}E_{p,q}(\varphi_t;D)
=\frac{1}{q}\int_D \frac{\partial}{\partial t}|\tau_p(\varphi_t)|^q\, v_g .
\end{equation}
Next, we compute the following term
\begin{eqnarray}\label{eq2.3}
\frac{\partial}{\partial t}|\tau_{p}(\varphi_{t})|^{q}
&=& \frac{\partial}{\partial t}\left[|\tau_{p}(\varphi_{t})|^{2}\right]^{\frac{q}{2}} \nonumber\\
&=& \frac{\partial}{\partial t}\left[h(\tau_{p}(\varphi_{t}),\tau_{p}(\varphi_{t}))\right]^{\frac{q}{2}} \nonumber\\
&=& \frac{q}{2}\,h(\tau_{p}(\varphi_{t}),\tau_{p}(\varphi_{t}))^{\frac{q}{2}-1}
\frac{\partial}{\partial t}h(\tau_{p}(\varphi_{t}),\tau_{p}(\varphi_{t})) \nonumber\\
&=& q\,|\tau_{p}(\varphi_{t})|^{\,q-2}
h\!\left(\nabla^{\phi}_{\frac{\partial}{\partial t}}\tau_{p}(\varphi_{t}),
\tau_{p}(\varphi_{t})\right).
\end{eqnarray}
Let $E_i=(0,e_i)$ for $i=1,\ldots,m$. Since $\nabla_{E_i}E_i=0$ at the point $x$, we obtain
\begin{eqnarray*}
\nabla^{\phi}_{\frac{\partial}{\partial t}} \tau_{p}(\varphi_{t})
&=& \nabla^{\phi}_{\frac{\partial}{\partial t}}
\Big[\nabla^{\phi}_{E_i}\big(|d\varphi_{t}|^{p-2}d\phi(E_i)\big)
-|d\varphi_{t}|^{p-2}d\phi(\nabla_{E_i}E_i)\Big] \\
&=& \nabla^{\phi}_{\frac{\partial}{\partial t}}
\nabla^{\phi}_{E_i}\big(|d\varphi_{t}|^{p-2}d\phi(E_i)\big).
\end{eqnarray*}
Using the definition of the curvature tensor of $(N,h)$, we obtain
\begin{eqnarray}\label{eq2.4}
\nabla_{\frac{\partial}{\partial t}}^{\phi} \tau_{p}(\varphi_{t})
&=& R^{N}\!\left(d\phi\!\left(\frac{\partial}{\partial t}\right),d\phi(E_{i})\right)
|d\varphi_{t}|^{p-2}d\phi(E_{i}) \nonumber\\
&&+\nabla_{E_{i}}^{\phi}\nabla_{\frac{\partial}{\partial t}}^{\phi}
\big(|d\varphi_{t}|^{p-2}d\phi(E_{i})\big)  \nonumber\\
&&+\nabla^{\phi}_{\left[\frac{\partial}{\partial t},E_{i}\right]}
\big(|d\varphi_{t}|^{p-2}d\phi(E_{i})\big).
\end{eqnarray}
Since $\left[\frac{\partial}{\partial t},E_i\right]=0$, it follows that
\begin{eqnarray}\label{eq2.4b}
\nabla_{\frac{\partial}{\partial t}}^{\phi} \tau_{p}(\varphi_{t})
&=& |d\varphi_{t}|^{p-2}
R^{N}\!\left(d\phi\!\left(\frac{\partial}{\partial t}\right),d\phi(E_{i})\right)d\phi(E_{i})
\nonumber\\
&&+\nabla_{E_{i}}^{\phi}\nabla_{\frac{\partial}{\partial t}}^{\phi}
\big(|d\varphi_{t}|^{p-2}d\phi(E_{i})\big).
\end{eqnarray}
Substituting \eqref{eq2.4b} into \eqref{eq2.3}, we obtain
\begin{eqnarray}\label{equa03}
\frac{1}{q}\frac{\partial}{\partial t}|\tau_{p}(\varphi_{t})|^{q}
&=& |d\varphi_{t}|^{p-2} |\tau_{p}(\varphi_{t})|^{q-2}
h\!\left(R^{N}\!\left(d\phi\!\left(\frac{\partial}{\partial t}\right),
d\phi(E_{i})\right)d\phi(E_{i}),\tau_{p}(\varphi_{t})\right) \nonumber\\
&&+ E_{i}\!\Big[|\tau_{p}(\varphi_{t})|^{q-2}
h\!\left(\nabla_{\frac{\partial}{\partial t}}^{\phi}
\big(|d\varphi_{t}|^{p-2}d\phi(E_{i})\big),\tau_{p}(\varphi_{t})\right)\Big] \nonumber\\
&&-h\!\left(\nabla_{\frac{\partial}{\partial t}}^{\phi}
\big(|d\varphi_{t}|^{p-2}d\phi(E_{i})\big),
\nabla_{E_{i}}^{\phi}\big(|\tau_{p}(\varphi_{t})|^{q-2}\tau_{p}(\varphi_{t})\big)\right).
\end{eqnarray}
Using the property $\nabla^{\phi}_{X}d\phi(Y)=\nabla^{\phi}_{Y}d\phi(X)+d\phi([X,Y])$, we obtain
\begin{eqnarray}\label{eq2.6}
\nabla_{\frac{\partial}{\partial t}}^{\phi}|d\varphi_t|^{p-2}d\phi(E_{i})\Big|_{t=0}
&=& \nabla_{\frac{\partial}{\partial t}}^{\phi}d\phi(|d\varphi_t|^{p-2}E_{i})\Big|_{t=0}\\
&=&\nonumber |d\varphi|^{p-2}\nabla^\varphi_{e_{i}}v+\frac{\partial}{\partial t}|d\varphi_t|^{p-2}\Big|_{t=0}d\varphi(e_i)\\
&=&\nonumber |d\varphi|^{p-2}\nabla^\varphi_{e_{i}}v
+(p-2)|d\varphi|^{p-4}h(\nabla^\varphi_{e_{j}}v,d\varphi(e_{j}))d\varphi(e_{i}).
\end{eqnarray}
Substituting (\ref{eq2.6}) into (\ref{equa03}), we obtain
\begin{eqnarray}\label{eq2.7}
   \frac{1}{q}\frac{\partial}{\partial t}|\tau_{p}(\varphi_{t})|^{q}\Big|_{t=0}
   &=&\nonumber |d\varphi|^{p-2} |\tau_{p}(\varphi)|^{q-2}h\big(R^{N}\big(v,d\varphi(e_{i})\big)d\varphi(e_{i}),\tau_{p}(\varphi) \big) \\
   &&\nonumber + e_{i}\Big[|d\varphi|^{p-2}|\tau_{p}(\varphi)|^{q-2}h(\nabla^\varphi_{e_{i}}v,\tau_{p}(\varphi))\Big]\\
   &&\nonumber + (p-2) e_{i}\Big[|d\varphi|^{p-4}|\tau_{p}(\varphi)|^{q-2}h(\nabla^\varphi_{e_{j}}v,d\varphi(e_{j}))h(d\varphi(e_{i}),\tau_{p}(\varphi))\Big]\nonumber\\
   &&\nonumber-|d\varphi|^{p-2}h(\nabla^\varphi_{e_{i}}v,\nabla_{e_{i}}^{\varphi}|\tau_{p}(\varphi)|^{q-2}\tau_{p}(\varphi_{t}))\\
   &&\nonumber-(p-2)|d\varphi|^{p-4}h(\nabla^\varphi_{e_{j}}v,d\varphi(e_{j}))
   h(d\varphi(e_{i}),\nabla_{e_{i}}^{\varphi}|\tau_{p}(\varphi)|^{q-2}\tau_{p}(\varphi)).
\end{eqnarray}
Let $\eta_{1},\eta_{2},\eta_{3},\eta_{4}\in \Gamma(T^{*}M)$ defined by
\begin{eqnarray*}
\eta_{1}(X)&=&|d\varphi|^{p-2}|\tau_p(\varphi)|^{q-2}h\big(\nabla_{X}^{\varphi}v,\tau_{p}(\varphi) \big),\\
\eta_{2}(X)&=&|d\varphi|^{p-4}|\tau_p(\varphi)|^{q-2}\left\langle\nabla^{\varphi}v,d\varphi\right\rangle h\big(d\varphi(X),\tau_{p}(\varphi)\big),\\
\eta_{3}(X)&=&|d\varphi|^{p-2}h\big(v,\nabla_{X}^{\varphi}|\tau_{p}(\varphi)|^{q-2}\tau_{p}(\varphi) \big),\\
\eta_{4}(X)&=&|d\varphi|^{p-4}
\left\langle\nabla^{\varphi}|\tau_{p}(\varphi)|^{q-2}\tau_{p}(\varphi),d\varphi\right\rangle h\big(v,d\varphi(X)\big).
\end{eqnarray*}
Therefore, equation \eqref{eq2.7} becomes
\begin{eqnarray*}
\frac{1}{q}\frac{\partial}{\partial t}|\tau_{p}(\varphi_{t})|^{q}\Big|_{t=0}
&=& |d\varphi|^{p-2}|\tau_{p}(\varphi)|^{q-2}
h\!\left(R^{N}\big(\tau_{p}(\varphi),d\varphi(e_{i})\big)d\varphi(e_{i}),v\right)\\
&&+ \operatorname{div}\eta_{1}+(p-2)\operatorname{div}\eta_{2}
-\operatorname{div}\eta_{3}-(p-2)\operatorname{div}\eta_{4} \\
&&+ h\!\left(v,\nabla_{e_{i}}^{\varphi}|d\varphi|^{p-2}
\nabla_{e_{i}}^{\varphi}|\tau_{p}(\varphi)|^{q-2}\tau_{p}(\varphi)\right) \\
&&+(p-2)h\!\left(v,\nabla_{e_{j}}^{\varphi}|d\varphi|^{p-4}
\left\langle\nabla^{\varphi}|\tau_{p}(\varphi)|^{q-2}\tau_{p}(\varphi),
d\varphi\right\rangle d\varphi(e_{j})\right).
\end{eqnarray*}
Applying the divergence theorem on the compact domain $D$, we obtain the formula \eqref{equa01}.
This completes the proof of Theorem \ref{th01}.
\end{proof}
\section{The $(p,q)$-harmonic hypersurfaces}
We present a characterization of $(p,q)$-harmonic hypersurfaces in terms of their mean curvature and geometric data.
\begin{theorem}\label{th2}
Let $(M^m, g)$ be a hypersurface in $(N^{m+1},h)$ with mean curvature vector $H = f \eta$. Then, $(M^m,g)$ is proper $(p,q)$-harmonic hypersurface if and only if $f\neq0$ and the following system of equations holds
\[
\left\{
\begin{array}{l}
\displaystyle
-(q-1)f\Delta^M(f) -(q-1)(q-2)|\operatorname{grad}^M f|^2
+ f^{2}|A|^2- f^{2} \, \mathrm{Ric}^N(\eta, \eta)\\[1.5ex]
\displaystyle
 + m(p - 2)f^{4} = 0,\\[2ex]
\displaystyle
2(q-1)A(\operatorname{grad}^M f) - 2 f (\mathrm{Ricci}^N \eta)^{\top} + f \big[ m
+(p-2)q \big] \operatorname{grad}^M f = 0,
\end{array}
\right.
\]
where $\mathrm{Ric}^N$ (resp. $\mathrm{Ricci}^N$) denotes the Ricci curvature (resp. Ricci tensor) of $(N^{m+1}, h)$, and $A$ is the shape operator of $M$ with respect to the unit normal vector field $\eta$.
\end{theorem}

\begin{proof}
Choose a local normal orthonormal frame $\{e_i\}_{i=1}^m$ on $(M^m, g)$ at $x$ such that $\{e_1, \dots, e_m, \eta\}$ forms an orthonormal frame of the ambient manifold $(N^{m+1}, h)$ along $M^m$. Let $\mathbf{i}: (M^m,g) \hookrightarrow (N^{m+1}, h)$
denote the canonical inclusion. Observe that $d\mathbf{i}(X) = X$, $\nabla_X^{\mathbf{i}} Y = \nabla_X^N Y$, $|d\mathbf{i}|^2=m$,
and the $p$-tension field of $\mathbf{i}$ satisfies $\tau_p(\mathbf{i}) = m^{\frac{p}{2}} f \, \eta$,
thus $|\tau_p(\mathbf{i})|^2=m^{p}f^2$.
We now proceed to compute the $p$-bitension field of $\mathbf{i}$. We have
\begin{equation}\label{th2-eq1}
\begin{aligned}
		\tau_{p, q}(\mathbf{i})= & -|d \mathbf{i}|^{p-2} |\tau_p(\mathbf{i})|^{q-2}\operatorname{trace}_g R^N\left(\tau_p(\mathbf{i}), d \mathbf{i}\right) d \mathbf{i}\\
		& -\operatorname{trace}_g \nabla^{\mathrm{i}}|d \mathbf{i}|^{p-2} \nabla^{\mathrm{i}} |\tau_p(\mathbf{i})|^{q-2}\tau_p(\mathbf{i})\\
		& -(p-2)\operatorname{trace}_g\nabla^{\mathrm{i}}|d \mathbf{i}|^{p-4}\left\langle\nabla^{\mathrm{i}} |\tau_p(\mathbf{i})|^{q-2}\tau_p(\mathbf{i}), d \mathbf{i}\right\rangle d \mathbf{i}.
	\end{aligned}
\end{equation}
The first term of (\ref{th2-eq1}) can be expressed as\\

$\displaystyle-|d\mathbf{i}|^{p-2}|\tau_p(\mathbf{i})|^{q-2}\operatorname{trace}_g R^N\left(\tau_p(\mathbf{i}), d \mathbf{i}\right) d \mathbf{i}$
\begin{equation}\label{th2-eq2}
\begin{aligned}
		\qquad\qquad\qquad&=-|d \mathbf{i}|^{p-2}|\tau_p(\mathbf{i})|^{q-2} \sum_{i=1}^m R^N\left(\tau_p(\mathbf{i}), d \mathbf{i}\left(e_i\right)\right) d \mathbf{i}\left(e_i\right)\\
		& =-m^{\frac{pq}{2}-1}f^{q-1} \sum_{i=1}^m R^N\left(\eta, e_i\right) e_i \\
		& = -m^{\frac{pq}{2}-1}f^{q-1}\operatorname{Ricci}^N \eta \\
		& =-m^{\frac{pq}{2}-1}f^{q-1}\left[\left(\operatorname{Ricci}^N \eta\right)^{\perp}+\left(\operatorname{Ricci}^N \eta\right)^{\top}\right].
	\end{aligned}
\end{equation}
We compute the second term of (\ref{th2-eq1}) at $x$. We have\\

$\displaystyle-\operatorname{trace}_g \nabla^{\mathbf{i}}|d \mathbf{i}|^{p-2} \nabla^{\mathbf{i}} |\tau_p(\mathbf{i})|^{q-2}\tau_p(\mathbf{i})$
\begin{align}\label{th2-eq3}
	\qquad\qquad\qquad&= -m^{\frac{pq}{2}-1} \sum_{i=1}^m \nabla_{e_i}^N \nabla_{e_i}^N f^{q-1} \eta \nonumber\\
	&= -m^{\frac{pq}{2}-1} \sum_{i=1}^m \nabla_{e_i}^N\left[(q-1)e_i(f)f^{q-2} \eta+f^{q-1} \nabla_{e_i}^N \eta\right] \nonumber\\
	&= -m^{\frac{pq}{2}-1}\Bigl[(q-1)\Delta^M(f)f^{q-2} \eta+(q-1)(q-2)f^{q-3}|\operatorname{grad}^{M}f|^2\eta \nonumber\\
	&\quad +2(q-1)f^{q-2} \nabla_{\operatorname{grad}^Mf}^N \eta+f^{q-1} \sum_{i=1}^m \nabla_{e_i}^N \nabla_{e_i}^N \eta\Bigr].
\end{align}
We compute the term $\displaystyle\sum_{i=1}^m \nabla_{e_i}^N \nabla_{e_i}^N \eta$. At $x$, we obtain
 \begin{eqnarray} \label{th2-eq4}
 \sum_{i=1}^m\nabla_{e_i}^N \nabla_{e_i}^N \eta
                                              &=& \nonumber \sum_{i=1}^m\nabla_{e_i}^N\left[(\nabla_{e_i}^N \eta)^\perp +(\nabla_{e_i}^N \eta)^\top \right] \\
                                              &=&\nonumber  -\sum_{i=1}^m\nabla_{e_i}^NA(e_i) \\
                                              &=& - \sum_{i=1}^m\nabla_{e_i}^M A(e_i)-\sum_{i=1}^mB(e_i , A(e_i)).
 \end{eqnarray}
 By using the property $g( A(X),Y) = h( B(X,Y) , \eta)  $, we get at $x$
 \begin{eqnarray}\label{th2-eq5}
 \nonumber \sum_{i=1}^m\nabla_{e_i}^M A(e_i)
                                 &=&\sum_{i,j=1}^m g( \nabla_{e_i}^M A(e_i), e_j)  e_j \\
 \nonumber                       &=&\sum_{i,j=1}^m\left[ e_i g( A(e_i) , e_j )  e_j - g( A(e_i) , \nabla^M_{e_i} e_j ) e_j \right]\\
 \nonumber                       &=&\sum_{i,j=1}^m e_i h( B(e_i,e_j) , \eta ) e_j \\
 \nonumber                       &=&\sum_{i,j=1}^m e_i h( \nabla_{e_j}^Ne_i , \eta ) e_j\\
                       &=&\sum_{i,j=1}^m h( \nabla_{e_i}^N\nabla_{e_j}^Ne_i , \eta ) e_j.
 \end{eqnarray}
By using the definition of curvature tensor $R^N$ of $(N^{m+1},h)$, we obtain
\begin{eqnarray} \label{th2-eq6}
\nonumber \sum_{i=1}^m\nabla_{e_i}^M A(e_i)
&=&\sum_{i,j=1}^m \left[\langle R^N (e_i,e_j)e_i , \eta\rangle e_j + h( \nabla_{e_j}^N\nabla_{e_i}^Ne_i , \eta) e_j\right] \\
\nonumber
&=&\sum_{i,j=1}^m \left[ -h( R^N (\eta , e_i) e_i ,e_j ) e_j +  h( \nabla_{e_j}^N\nabla_{e_i}^Ne_i , \eta ) e_j \right]\\
\nonumber
&=&  - \sum_{j=1}^m h( \operatorname{Ricci}^N \eta , e_j ) e_j + \sum_{i,j=1}^m e_j h( \nabla_{e_i}^Ne_i , \eta) e_j\\
& &-\sum_{i,j=1}^m h( \nabla_{e_i}^N{e_i}, \nabla_{e_i}^N \eta ) e_j \nonumber\\
&=& -( \operatorname{Ricci}^N \eta )^\top + m \operatorname{grad}^M f.
\end{eqnarray}
On the other hand, we have the following
 \begin{eqnarray} \label{th2-eq7}
 \nonumber \sum_{i=1}^mB(e_i,A(e_i )) &=& \sum_{i=1}^m h( B(e_i , A(e_i)) , \eta)   \eta \\
 \nonumber                &=& \sum_{i=1}^mg(  A(e_i) , A(e_i ))  \eta \\
                          &=& |A|^2 \eta.
 \end{eqnarray}
Substituting (\ref{th2-eq6}) and (\ref{th2-eq7}) in (\ref{th2-eq4}), we obtain at $x$
\begin{equation}\label{th2-eq8}
\sum_{i=1}^m \nabla_{e_i}^N \nabla_{e_i}^N \eta  =\left(\operatorname{Ricci}^{N} \eta\right)^{\top}-m\operatorname{grad}^{M}f-|A|^2\eta.
\end{equation}
Substituting (\ref{th2-eq8}) in (\ref{th2-eq3}), we conclude that\\

$\displaystyle-\operatorname{trace}_g \nabla^{\mathbf{i}}|d \mathbf{i}|^{p-2} \nabla^{\mathbf{i}} |\tau_p(\mathbf{i})|^{q-2}\tau_p(\mathbf{i})$
\begin{align}\label{th2-eq9}
	\qquad\qquad\qquad&= -m^{\frac{pq}{2}-1}\Bigl[(q-1)\Delta^M(f)f^{q-2} \eta+(q-1)(q-2)f^{q-3}|\operatorname{grad}^{M}f|^2\eta \nonumber\\
	&\quad -2(q-1)f^{q-2} A(\operatorname{grad}^Mf)+f^{q-1} \left(\operatorname{Ricci}^{N} \eta\right)^{\top}\nonumber\\
    &\quad -m f^{q-1}\operatorname{grad}^{M}f-f^{q-1}|A|^2\eta\Bigr].
\end{align}
The third term of (\ref{th2-eq1}) is provided by\\

$\displaystyle-(p-2) \operatorname{trace}_g \nabla^{\mathbf{i}}|d \mathbf{i}|^{p-4}\left\langle\nabla^{\mathbf{i}} |\tau_p(\mathbf{i})|^{q-2}\tau_p(\mathbf{i}), d \mathbf{i}\right\rangle d \mathbf{i}$
\begin{equation}\label{th2-eq10}
\begin{aligned}
&=-(p-2) m^{\frac{pq-4}{2}} \sum_{i, j=1}^m \nabla_{e_j}^N h( \nabla_{\mathrm{e}_i}^N f^{q-1} \eta, e_i) e_j.
\end{aligned}
\end{equation}
Furthermore, we have
\begin{equation}\label{th2-eq11}
\begin{aligned}
& \sum_{i=1}^m\left\langle\nabla_{e_i}^N f^{q-1} \eta, e_i\right\rangle=\sum_{i=1}	
^m h(f^{q-1}\nabla_{e_i}^N \eta,e_i)
=-m f^q.
\end{aligned}
\end{equation}
Substituting (\ref{th2-eq11}) in (\ref{th2-eq10}), we find that\\

$\displaystyle-(p-2) \operatorname{trace}_g \nabla^{\mathbf{i}}|d \mathbf{i}|^{p-4}\left\langle\nabla^{\mathbf{i}} |\tau_p(\mathbf{i})|^{q-2}\tau_p(\mathbf{i}), d \mathbf{i}\right\rangle d \mathbf{i}$
\begin{equation}\label{th2-eq12}
\begin{aligned}
\qquad\qquad\qquad&=m^{\frac{pq}{2}-1}(p-2)\left[qf^{q-1}\operatorname{grad}^M f+m f^{q+1} \eta\right].
\end{aligned}
\end{equation}
The Theorem \ref{th2} follows by equations (\ref{th2-eq1}), (\ref{th2-eq2}), (\ref{th2-eq9}), and (\ref{th2-eq12}).
\end{proof}

\begin{remark}
Theorem \ref{th2} can be regarded as a natural generalization of Ou's biharmonic hypersurface characterization in \cite{OU}. In particular, setting $p=2$ and $q=2$ recovers the system \eqref{S}. The introduction of the two real parameters $p$ and $q$ can be interpreted as a two-parameter perturbation of Ou's biharmonic equation.
\end{remark}

As a direct application of Theorem \ref{th2}, we obtain a more explicit characterization in the case where the ambient manifold is a space form. Indeed, using the curvature identity $R^N(X,Y)Z = c\big[h(Y,Z)X - h(X,Z)Y\big]$ and the relations $\mathrm{Ric}^N(\eta,\eta)=mc$,
$(\operatorname{Ricci}^{N} \eta)^{\top}=0$, the general $(p,q)$-harmonic equations reduce to the following system.

\begin{proposition}\label{pro1}
Let $(M^m, g)$ be a hypersurface in a space form $(N^{m+1}(c),h)$ with mean curvature vector $H = f \eta$. Then, $(M^m,g)$ is proper $(p,q)$-harmonic hypersurface if and only if $f\neq0$ and the following system of equations holds
\[
\left\{
\begin{array}{l}
\displaystyle
-(q-1)\left[f\Delta^M(f)+(q-2)|\operatorname{grad}^M f|^2\right]
+ \left(|A|^2-mc\right) f^{2}+ m(p - 2)f^{4} = 0, \\[2ex]
\displaystyle
2(q-1)A(\operatorname{grad}^M f)  + f \big[ m
+(p-2)q \big] \operatorname{grad}^M f = 0.
\end{array}
\right.
\]
\end{proposition}

As a direct consequence of Theorem \ref{th2}, we obtain a simplified characterization in the case where the ambient manifold is Einstein.

\begin{corollary}
A hypersurface $(M^m, g)$ immersed in an Einstein manifold $(N^{m+1}, h)$ is $(p,q)$-harmonic if and only if its mean curvature function $f$ satisfies the following system of partial differential equations
\[
\left\{
\begin{array}{l}
\displaystyle
-(q-1)f\Delta^M(f) -(q-1)(q-2)|\operatorname{grad}^M f|^2
+ f^{2}|A|^2- \frac{S f^{2}}{m+1}\\[1.5ex]
\displaystyle
 + m(p - 2)f^{4} = 0,\\[2ex]
\displaystyle
2(q-1)A(\operatorname{grad}^M f)  + f \big[ m
+(p-2)q \big] \operatorname{grad}^M f = 0,
\end{array}
\right.
\]
where $S$ is the scalar curvature of the ambient space.
\end{corollary}

\begin{proof}
It is well known that if $(N^{m+1}, h)$ is an Einstein manifold, then its Ricci tensor satisfies
$\mathrm{Ric}^N(X,Y) = \lambda\, h(X,Y)$
for all $X, Y \in \Gamma(TN)$, where $\lambda$ is a constant. Therefore,
\[
\begin{aligned}
	S &= \mathrm{trace}_{h} \mathrm{Ric}^N \\
	&= \sum_{i=1}^m \mathrm{Ric}^N(e_i, e_i) + \mathrm{Ric}^N(\eta, \eta) \\
	&= \lambda(m + 1),
\end{aligned}
\]
where $\{e_i\}_{i=1}^m$ is a local orthonormal frame on $(M^m, g)$. Since $\mathrm{Ric}^N(\eta,\eta) = \lambda$, it follows that
\[
\mathrm{Ric}^N(\eta,\eta) = \frac{S}{m+1}.
\]
On the other hand, we compute
\[
\begin{aligned}
(\mathrm{Ricci}^N \eta)^\top
&= \sum_{i=1}^m h( \mathrm{Ricci}^N \eta, e_i ) e_i \\
&= \sum_{i=1}^m \mathrm{Ric}^N(\eta, e_i)\, e_i \\
&= \sum_{i=1}^m \lambda h( \eta, e_i ) e_i \\
&= 0.
\end{aligned}
\]
The result then follows from Theorem \ref{th2}.
\end{proof}

As a further consequence, in the case of totally umbilical hypersurfaces, the $(p,q)$-harmonic condition reduces to minimality under a curvature assumption on the ambient space.

\begin{corollary}
A totally umbilical hypersurface $(M^m, g)$ immersed in an Einstein manifold $(N^{m+1}, h)$ with non-positive scalar curvature is $(p,q)$-harmonic if and only if it is minimal.
\end{corollary}

\begin{proof}
It is well known that for a totally umbilical hypersurface $(M^m, g)$ in an Einstein manifold $(N^{m+1}, h)$, we have the shape operator
$A = f\mathrm{Id}$, where $\mathrm{Id}$ is the identity map. The $(p,q)$-harmonic hypersurface equations becomes
\[
\left\{
\begin{array}{l}
\displaystyle
-(q-1)f\Delta^M(f) -(q-1)(q-2)|\operatorname{grad}^M f|^2
+m(p-1) f^{4}- \frac{S f^{2}}{m+1}=0,\\[2ex]
\displaystyle
f\left[ m+(p-2)q+2(q-1) \right]\operatorname{grad}^M f = 0.
\end{array}
\right.
\]	
Solving this system, we obtain either $f = 0$, or
\[
f = \pm \sqrt{\frac{S}{m(m+1)(p-1)}},
\]
In the latter case, $f$ is constant, which occurs only when $S \geq 0$. This completes the proof.
\end{proof}

We next provide an example demonstrating the existence of totally umbilical proper $(p,q)$-harmonic hypersurfaces in Einstein manifolds with positive scalar curvature. In particular, the standard Euclidean sphere offers a natural setting for explicitly constructing such hypersurfaces. The example below illustrates that certain hyperspheres in $\mathbb{S}^{m+1}$ possess constant mean curvature and fulfill the $(p,q)$-harmonic condition.

\begin{example}
	 We consider the following hypersurface
	$$
	\mathbb{S}^m(a)=\left\{\left(x^1, \cdots, x^m, x^{m+1}, b\right) \in \mathbb{R}^{m+2},\, \sum_{i=1}^{m+1}\left(x^i\right)^2=a^2\right\} \subset \mathbb{S}^{m+1}
	$$
	where $a^2+b^2=1$ with $a,b\neq0$. A straightforward calculation shows that
	$$
	\eta=\frac{1}{r}\left(x^1, \cdots, x^{m+1},-\frac{a^2}{b}\right),
	$$
	with $r^2=\frac{a^2}{b^2}(r>0)$, defines a unit normal vector field along $\mathbb{S}^m(a)$ in $\mathbb{S}^{m+1}$.
	Let $X \in \Gamma\left(T \mathbb{S}^m(a)\right)$, we find that
	$$
	\nabla_X^{\mathbb{S}^{m+1}} \eta=\frac{1}{r} \nabla_X^{\mathbb{R}^{m+2}}\left(x^1, \cdots, x^{m+1},-\frac{a^2}{b}\right)=\frac{1}{r} X .
	$$
Hence, we have $A=-\frac{1}{r}\, \mathrm{Id}$, which gives $f=-\frac{1}{r}$, so that $\mathbb{S}^m(a)$ has constant mean curvature in $\mathbb{S}^{m+1}$. Since $|A|^2 = \frac{m}{r^2}$, Proposition \ref{pro1} implies that $\mathbb{S}^m(a)$ is a proper $(p,q)$-harmonic hypersurface in $\mathbb{S}^{m+1}$ if and only if
\[
\left(\frac{b^2}{a^2}-1\right)+(p-2)\frac{b^2}{a^2}=0.
\]
which is equivalent to $p =1/b^2$.
\end{example}

We now give an example where the mean curvature function is non-constant.

\begin{example}
Consider a surface of revolution $M$ in $\mathbb{R}^3$ given by
\[
X(u,v) = \big(r\,u\cos v,\, r\,u\sin v,\, u\big),
\]
where $r>0$ is a constant, for $u>0$ and $v\in(0,2\pi)$. Note that the induced Riemannian metric on $M$ is
$g=(1+r^2)\,du^2+r^2\,u^2\,dv^2$. A direct computation shows that the mean curvature function $f=f(u)$ is given by
\[
f(u) = \frac{1}{2\,r\sqrt{1 + r^2}}\,\frac{1}{u},
\]
where the unit normal vector field along $M$ in $\mathbb{R}^{3}$ is defined by
$$\eta=-\frac{\cos v}{\sqrt{1 + r^2}}\,\frac{\partial}{\partial x}
       -\frac{\sin v}{\sqrt{1 + r^2}}\,\frac{\partial}{\partial y}
       +\frac{r}{\sqrt{1 + r^2}}\,\frac{\partial}{\partial z}.$$
A straightforward computation yields
\begin{eqnarray*}
 \Delta^M(f)  &=& \frac{1}{2\,r\,(1 + r^2)^{\frac{3}{2}}}\,\frac{1}{u^3},  \\
 \operatorname{grad}^M f  &=& -\frac{1}{2\,r\,(1 + r^2)^{\frac{3}{2}}}\,\frac{1}{u^2}\,\frac{\partial}{\partial u},\\
 |\operatorname{grad}^M f|^2  &=& \frac{1}{4\,r^2\,(1 + r^2)^{2}}\,\frac{1}{u^4},\\
 |A|^2 &=& \frac{1}{r^2\,(1 + r^2)}\,\frac{1}{u^2},\\
 A(\operatorname{grad}^M f) &=& 0.
\end{eqnarray*}
Substituting these expressions into the system given in Proposition \ref{pro1}, We obtain the system of algebraic equations
$-2\,{r}^{2}{q}^{2}+4\,{r}^{2}\,q-2\,{r}^{2}+p=0$ and $pq-2q+2=0$. We can solve for parameters $p$ and $q$, we get the following
$$p=2\left(1-\frac{1}{q}\right),\quad r=\frac{1}{\sqrt{q\left(q-1\right)}}.$$
Thus, for this choice the surface $M$ is $(p,q)$-harmonic with a non-constant mean curvature function. This gives an explicit example of a proper $(p,q)$-harmonic surface in $\mathbb{R}^3$.
\end{example}


\subsection{$(p,q)$-harmonic curves in $3$-dimensional space form}
Let $(N^3(c),h)$ be a three-dimensional Riemannian manifold of constant curvature $c$, and let
$\gamma: I \longrightarrow (N^3(c),h)$ be a curve parametrized by arc length, where $I \subseteq \mathbb{R}$ is an open interval. Consider an orthonormal frame field $\{T, N, B\}$ along $\gamma$, with $T = \frac{d\gamma}{dt}$ the unit tangent vector, $N$ the unit normal vector in the direction of $\nabla_T T$, and $B$ chosen so that $\{T, N, B\}$ forms a positively oriented basis. Denote by $\nabla$ the Levi-Civita connection of $(N^3(c),h)$. Then, the Frenet equations along $\gamma$ are given by
\[
\left\{
\begin{aligned}
\nabla_T T &= k N,\\
\nabla_T N &= -k T + \tau B,\\
\nabla_T B &= -\tau N,
\end{aligned}
\right.
\]
where $k$ and $\tau$ are the geodesic curvature and geodesic torsion of $\gamma$, respectively.
Recalling that the tension field of $\gamma$ is given by $\tau(\gamma) = \nabla_T T = k N$.\\
The following theorem provides a characterization of proper $(p,q)$-biharmonic curves in space forms.
\begin{theorem}\label{th4}
A curve $\gamma: I \longrightarrow (N^3(c), h)$ is proper $(p,q)$-harmonic if and only if its geodesic curvature $k$ and torsion $\tau$ are constant with $k>0$, and the parameter $p$ satisfies $\displaystyle p=\frac{c-\tau^{2}}{k^{2}}+1$.
\end{theorem}
\begin{proof}
From formula (\ref{eq1.1B}), it follows that
\begin{equation}\label{th4-eq1}
\tau_{p}(\gamma)=| d\gamma|^{p-2}\tau(\gamma) + d\gamma(\operatorname{grad}^I| d\gamma|^{p-2}).
\end{equation}
Note that, the Hilbert-Schmidt norm of $d\gamma$ is $|d\gamma|^{2}=h(T,T)=1$.
Thus, $\tau_{p}(\gamma)=kN$. We now proceed to compute the $(p,q)$-tension field of $\gamma$, which yields
\begin{eqnarray}\label{th4-eq2}
\tau_{p,q}(\gamma)
&=&\nonumber-|\tau_{p}(\gamma)|^{q-2}\operatorname{trace}_gR\big(\tau_{p}(\gamma),d\gamma\big)d\gamma\\
&&\nonumber- \operatorname{trace}_g\nabla^{\gamma}\nabla^{\gamma}|\tau_{p}(\gamma)|^{q-2}\tau_{p}(\gamma)\\
&&-(p-2)\operatorname{trace}_g\nabla^{\gamma}
\left\langle\nabla^{\gamma}|\tau_{p}(\gamma)|^{q-2}\tau_{p}(\gamma),d\gamma\right\rangle d\gamma,
\end{eqnarray}
where $g = dt^2$ and $R$ denotes the Riemannian curvature tensor of $(N^3(c),h)$.
Considering the first term of (\ref{th4-eq2}), we have
\begin{eqnarray}\label{th4-eq3}
\operatorname{trace}_{g}R(\tau_p(\gamma),d\gamma)d\gamma
&=&   k R(N,T)T=ckN.
\end{eqnarray}
For the second term of (\ref{th4-eq2}), we compute\\

$\displaystyle\operatorname{trace}_{g}\nabla^{\gamma}\nabla^{\gamma}|\tau_{p}(\gamma)|^{q-2}\tau_p(\gamma)$
\begin{eqnarray}\label{th4-eq4}
 \quad\quad\quad&=& \nabla^{\gamma}_{\frac{d}{dt}}\nabla^{\gamma}_{\frac{d}{dt}}k^{q-1}N \\ \nonumber
 &=&\nabla^{\gamma}_{\frac{d}{dt}}[(q-1)k^{q-2}k'N+k^{q-1}\nabla^{\gamma}_{\frac{d}{dt}}N]\\ \nonumber
 &=&\nabla^{\gamma}_{\frac{d}{dt}}[(q-1)k^{q-2}k'N+k^{q-1}\nabla_{T}N]\\ \nonumber
 &=& \nabla^{\gamma}_{\frac{d}{dt}}[(q-1)k^{q-2}k'N-k^{q}T + k^{q-1}\tau B] \\ \nonumber
 &=& (q-1)(q-2)k^{q-3}(k')^2N+(q-1)k^{q-2}k''N+(q-1)k^{q-2}k'\nabla_TN\\ \nonumber
 & & -qk^{q-1}k'T-k^{q}\nabla_TT
     +(q-1) k^{q-2}k'\tau B+ k^{q-1}\tau' B+ k^{q-1}\tau \nabla_TB.
\end{eqnarray}
Applying the Frenet equations, equation (\ref{th4-eq4}) takes the form\\

$\displaystyle\operatorname{trace}_{g}\nabla^{\gamma}\nabla^{\gamma}|\tau_{p}(\gamma)|^{q-2}\tau_p(\gamma)$
\begin{eqnarray}\label{th4-eq5}
 \quad\quad\quad&=& \nonumber
     (q-1)(q-2)k^{q-3}(k')^2N+(q-1)k^{q-2}k''N+(q-1)k^{q-2}k'(-k T + \tau B)\\
   &&  -qk^{q-1}k'T-k^{q+1} N
     +(q-1) k^{q-2}k'\tau B+ k^{q-1}\tau' B- k^{q-1}\tau^2 N.
\end{eqnarray}
We now compute the following term
\begin{eqnarray*}
   \left\langle\nabla^{\gamma}|\tau_{p}(\gamma)|^{q-2}\tau_{p}(\gamma),d\gamma\right\rangle
   &=&h(\nabla^{\gamma}_{\frac{d}{dt}}k^{q-1}N,d\gamma(\frac{d}{dt}))\\
   &=& h((q-1)k^{q-2}k'N-k^{q}T+k^{q-1}\tau B,T) \\
   &=& -k^{q}.
\end{eqnarray*}
Hence, for the third term in (\ref{th4-eq2}), we obtain
\begin{eqnarray}\label{th4-eq6}
\operatorname{trace}_g\nabla^{\gamma}
\left\langle\nabla^{\gamma}|\tau_{p}(\gamma)|^{q-2}\tau_{p}(\gamma),d\gamma\right\rangle d\gamma
   &=&\nonumber -\nabla^{\gamma}_{\frac{d}{dt}} k^q d\gamma(\frac{d}{dt}) \\
   &=&\nonumber -(qk^{q-1}k' T+k^q\nabla_TT) \\
   &=& -qk^{q-1}k'T-k^{q+1}N.
\end{eqnarray}
Substituting equations (\ref{th4-eq3}), (\ref{th4-eq5}), and (\ref{th4-eq6}) into (\ref{th4-eq2}), we conclude that
$\gamma$ is $(p,q)$-biharmonic if and only if the following system holds
\begin{equation}\label{SYS}
  \left\{
    \begin{array}{ll}
       & (pq-1)k^{q-1}k'=0,  \\ [2ex]
       & ck^{q-1}+(q-1)(q-2)k^{q-3}(k')^2+(q-1)k^{q-2}k''\\ [1.5ex]
       & -k^{q+1}-k^{q-1}\tau^2-(p-2)k^{q+1}=0,\\ [2ex]
       & (q-1)k^{q-2}k'\tau +(q-1) k^{q-2}k'\tau + k^{q-1}\tau' =0.
    \end{array}
  \right.
\end{equation}
Since $pq>1$ and $k \neq 0$, the first equation of (\ref{SYS}) implies that $k$ is constant on $I$.
Moreover, the third equation of (\ref{SYS}) shows that $\tau$ is also constant on $I$.
Finally, the conclusion of Theorem~\ref{th4} follows from the second equation of (\ref{SYS}).
\end{proof}

In the particular cases where $N^3(c)$ is the Euclidean $3$-space $(c=0)$, the hyperbolic $3$-space $(c=-1)$, or the standard $3$-sphere $(c=1)$, Theorem \ref{th4} yields the following results.

\begin{corollary}
Then there exist no  $(p,q)$-harmonic curves with constant curvature $k\neq0$ and torsion $\tau\neq0$ in $\mathbb{R}^3$ or $\mathbb{H}^3$.
\end{corollary}

\begin{corollary}\label{co2}
Let $\gamma: I \subset \mathbb{R} \longrightarrow (\mathbb{S}^3, h)$ be a curve with constant curvature $k$ and torsion $\tau$.
Then $\gamma$ is a proper $(p,q)$-harmonic curve if and only if $k>0$ and $\tau^2 < 1$, where
\[
p = \frac{1 - \tau^2}{k^2} + 1.
\]
\end{corollary}

A helix in $\mathbb{S}^3$ is a curve with constant geodesic curvature and torsion. The following example provides an explicit parametrization of such a curve and determines the conditions under which it is proper $(p,q)$-harmonic.

\begin{example}
An arbitrary helix in $\mathbb{S}^3$ can be parametrized by
\[
\gamma(t) = \big(\cos(\alpha)\cos(at), \cos(\alpha)\sin(at), \sin(\alpha)\cos(bt), \sin(\alpha)\sin(bt)\big),
\]
where $\alpha \in (0,\pi/2)$ and $a,b$ are positive real numbers. Assume that
\[
a^2\cos^2(\alpha) + b^2\sin^2(\alpha) = 1,
\]
which guarantees that $|\gamma'(t)| = 1$. Choosing $a>b$, the geodesic curvature $k$ and torsion $\tau$ of $\gamma$ are given by
\[
k = \sqrt{(a^2 - 1)(1 - b^2)}, \qquad \tau = ab
\]
(see \cite{Branding}). According to Corollary \ref{co2}, and under the assumption that $p,q>1$, the curve $\gamma$ is proper $(p,q)$-harmonic if and only if
\[
a,b \neq 1,\quad(ab)^2<1,\quad
p = \frac{a^2 + b^2 - 2a^2 b^2}{(a^2 - 1)(1 - b^2)}.
\]
\end{example}


\subsection*{Author contributions} The authors have reviewed the manuscript.

\subsection*{Data availability Statement}  Not applicable.

\subsection*{Declarations Conflicts of Interest} The authors declare no conflict of interest.


\begin{thebibliography}{99}



\bibitem{BG} Baird, P., Gudmundsson, S.: {\it $p$-Harmonic maps and minimal submanifolds}. Math. Ann. \textbf{294}, 611--624 (1992).

\bibitem{baird} Baird, P., Wood, J. C.: {\it Harmonic morphisms between Riemannian manifolds}. Clarendon Press, Oxford (2003).

\bibitem{BI} Bojarski, B., Iwaniec, T.: {\it $p$-Harmonic equation and quasiregular mappings}. Banach Center Publ. \textbf{19}(1), 25--38 (1987).

\bibitem{Branding} Branding, V.: {\it On polyharmonic helices in space forms}. Arch. Math. \textbf{120}, 213--225 (2023).


\bibitem{CL2} Cao, X., Luo, Y.: {\it On $p$-biharmonic submanifolds in nonpositively curved manifolds}. Kodai Math. J. \textbf{39}, 567--578 (2016).

\bibitem{ES} Eells, J., Sampson, J. H.: {\it Harmonic mappings of Riemannian manifolds}. Amer. J. Math. \textbf{86}, 109--160 (1964).

\bibitem{ali} Fardoun, A.: {\it On equivariant $p$-harmonic maps}. Ann. Inst. H. Poincar\'{e} \textbf{15}, 25--72 (1998).


\bibitem{Jiang} Jiang, G. Y.: {\it $2$-harmonic maps and their first and second variational formulas}. Chinese Ann. Math. Ser. A \textbf{7}(4), 389--402 (1986).

\bibitem{LC} Latti, F., Mohammed Cherif, A.: {\it On the generalized of $p$-biharmonic and bi-$p$-harmonic maps}. arXiv:2603.06133.



\bibitem{cherif1} Mohammed Cherif, A.: {\it Liouville type theorems for generalized $p$-harmonic maps}. Arab. J. Math. \textbf{13}, 255--262 (2024).

\bibitem{MM2} Mohammed Cherif, A., Mouffoki, K.: {\it $p$-Biharmonic hypersurfaces in Einstein space and conformally flat space}. Bull. Korean Math. Soc. \textbf{60}, 705--715 (2023).

\bibitem{cherif2} Mohammed Cherif, A.: {\it On the $p$-harmonic and $p$-biharmonic maps}. J. Geom. \textbf{109}, 41 (2018).


\bibitem{NU} Nakauchi, N., Urakawa, H.: {\it Biharmonic hypersurfaces in a Riemannian manifold with non-positive Ricci curvature}. Ann. Global Anal. Geom. \textbf{40}, 125--131 (2011).

\bibitem{ON} O'Neill, B.: {\it Semi-Riemannian Geometry}. Academic Press, New York (1983).


\bibitem{OU} Ou, Y.-L.: {\it Biharmonic hypersurfaces in Riemannian manifolds}. Pacific J. Math. \textbf{248}(1), 217--232 (2010).

\bibitem{YX} Xin, Y.: {\it Geometry of harmonic maps}. Fudan University (1996).

\end{thebibliography}
\end{document}